\documentclass[12pt]{amsart}

\usepackage{amssymb,latexsym}
\usepackage{amsmath}               
\usepackage{amsfonts}       
\usepackage{amsthm}                
\usepackage{setspace}
\usepackage{amssymb}
\usepackage{mathtools}

\newtheorem{theorem}{Theorem}[section]
\newtheorem{lemma}[theorem]{Lemma}
\newtheorem{proposition}[theorem]{Proposition}
\newtheorem{corollary}[theorem]{Corollary}

\def\N{\mathbb{N}}
\def\Z{\mathbb{Z}}
\def\Q{\mathbb{Q}}

\begin{document}

\baselineskip=17pt


\title[]{Non-extensibility of the pair $\{1, 3\}$ to a
Diophantine quintuple in $\Z\left[\sqrt{-d}\right]$}

\author[D. Kreso]{Dijana Kreso}
\address{Institut f\"ur Analysis und Computational Number Theory (Math A)\\
Technische Universit\"at Graz\\ Steyrergasse 30/II\\
8010 Graz, Austria}
\email{kreso@math.tugraz.at}

\author[Z. Franu\v si\' c]{Zrinka Franu\v si\' c}
\address{Department of Mathematics,
University of Zagreb, Croatia\\
Bijeni\v{c}ka c.\@ 30, 10 000 Zagreb, Croatia}
\email{fran@math.hr}

\begin{abstract}
We show that the Diophantine pair $\{1, 3\}$ can not be extended
to a Diophantine quintuple in the ring $\Z\left[\sqrt{-2}\right]$. This result completes the
work of the first author and establishes non-extensibility of the Diophantine pair
$\{1, 3\}$ to a Diophantine quintuple in $\Z\left[\sqrt{-d}\right]$ for all $d\in \N$.
\end{abstract}

\maketitle

\section{Introduction and results}
Let $R$ be a commutative ring with unity $1$. The set
$\{a_1,a_2,\ldots,a_m\}$ in $R$ such that $a_i\not =0$ for all
$i=1,\ldots,m$, $a_i\not =a_j$ and $a_ia_j+1$ is a square in $R$ for
all $1\leq i<j\leq m$, is called a {\it Diophantine $m$-tuple} in $R$.
 The problem of constructing such sets was first studied by
Diophantus of Alexandria who found a set of four rationals
$\left \{\frac{1}{16},\frac{33}{16},\frac{17}{4},\frac{105}{16}\right \}$ with
the given property. Fermat found a first Diophantine quadruple in integers
-  the set $\{1,3,8,120\}$. A Diophantine pair $\{a,b\}$
in a ring $R$, which satisfies $ab+1=r^2$, can be extended to a
Diophantine quadruple in $R$ by adding elements $a + b + 2r$ and
$4r(r + a)(r + b)$, provided all four elements are nonzero and different. 
Hence, apart from some exceptional cases, Diophantine quadruples in 
a ring $R$ exist, but can we obtain Diophantine $m$-tuples of size greater 
than $4$?

The folklore conjecture is that there are no
Diophantine quintuples in integers. In 1969, Baker and Davenport \cite{BD69}
showed that the set $\{1,3,8\}$ can not be extended to a Diophantine
quintuple, which was the first result supporting the conjecture.
This result was first generalized by Dujella \cite{D97}, who showed
that the set $\{k-1, k+1, 4k\}$, with integer $k\geq 2$, can not be extended to
a Diophantine quintuple in $\Z$. Dujella and Peth\H{o} \cite{DP98} later showed
that not even the Diophantine pair $\{1,3\}$
can be extended to a Diophantine quintuple in $\Z$. Greatest step
towards proving the conjecture did Dujella \cite{D04} in 2004; he
showed that there are no Diophantine sextuples in $\Z$ and that there are
only finitely many Diophantine quintuples. In \cite{DF04} it was
proved that there are no Diophantine quintuples in the ring of
polynomials with integers coefficients under assumption that not all
elements are constant polynomials.

The size of Diophantine $m$-tuples can be greater than $4$ in some rings. For
instance, the set 
$$\left \{\frac{11}{192}, \frac{35}{192}, \frac{155}{27},
\frac{512}{27}, \frac{1235}{48}, \frac{180873}{16}\right \}$$
is a Diophantine sextuple in $\Q$; it was found by Gibbs \cite{G06}.
Furthermore, we can construct Diophantine quintuples in the ring
$\Z\big[\sqrt{d}\big]$ for some values of $d$; for instance
$\{1,3,8,120,1678\}$ is a Diophantine quintuple in
$\Z\left[\sqrt{201361}\right]$. It is natural to start investigating the upper
bound for the size of Diophantine $m$-tuples in $\Z\big[\sqrt{d}\big]$ by
focusing on a problem of extensibility of Diophantine triples
$\{k-1,k+1, 4k\}$ and Diophantine pair
$\{1,3\}$ to a Diophantine quintuple in $\Z\big[\sqrt{d}\big]$, since the
problem in integers was approached similarly, see \cite{DP98} and
\cite{D97}.

In \cite{F10} Franu\v si\' c proved that the Diophantine pair
$\{1,3\}$ can not be extended to a Diophantine quintuple in $\Z\left[\sqrt{-d}\right]$ if
$d$ is a positive integer and $d\not =2$. The case $d=2$ was also
considered and it was shown that if $\{1,3,c\}$ is a Diophantine
triple in $\Z\left[\sqrt{-2}\right]$, then $c\in \{c_k,d_l\}$, where the
sequences $(c_k)$ and $(d_l)$ are given by
\begin{align}
\label{ck} c_k&=\frac{1}{6}\big((2+\sqrt{3})(7+4\sqrt{3})^k+(2-\sqrt{3})(7-4\sqrt{3})^k-4\big), \\ 
\label{dk} d_l&=\frac{-1}{6}\big((7+4\sqrt{3})^l+(7-4\sqrt{3})^l+4\big), 
\end{align}
where $k\geq 1$ and $l \geq 0$.
Sequences $(c_k)$ and $(d_l)$ are defined recursively as follows
\begin{align}
c_0&=0, & c_1&=8, & c_{k+2}&=14c_{k+1}-c_{k}+6;\\
d_0 &=-1, & d_1 &=-3, & d_{l+2}&=14d_{l+1}-d_{l}+8.
\end{align}
It is known that $\{1,3,c_k,c_{k+1}\}$, with $k\geq 1$, is a Diophantine
quadruple in $\Z$, see \cite{DP98}, and hence also in $\Z\left[\sqrt{-2}\right]$. The
set $\{1,3,d_l,d_{l+1}\}$ is a Diophantine quadruple in
$\Z\left[\sqrt{-2}\right]$ since
\begin{equation}\label{e.100}
d_ld_{l+1}+1=(c_l+2)^2
\end{equation}
for every $l\geq 0$; this easily follows from identities \eqref{ck} and \eqref{dk}. The set
$\{1,3,c_k,d_l\}$ is not a Diophantine quadruple for $k\geq 1$ and
$l\geq 0$ since $1+c_kd_l$ is a negative odd number and hence it can
not be a square in $\Z\left[\sqrt{-2}\right]$. Therefore, if there is an
extension of the Diophantine pair $\{1, 3\}$ to a Diophantine
quadruple in $\Z\left[\sqrt{-2}\right]$, then it is of the form
$\{1,3,c_k,c_l\}$, with $l>k\geq 1$ or $\{1,3,d_k,d_l\}$, with $l>k\geq 0$. In
the former case, the set can not be extended to a Diophantine
quintuple in $\Z$, see \cite{DP98}, wherefrom it easily follows that it can not be extended
to a Diophantine quintuple in  $\Z\left[\sqrt{-2}\right]$. It remains to
examine the latter case. We can formulate the following theorem.

\begin{theorem}\label{trm1}
Let $k$ be a nonnegative integer and $d$ an integer. If the set $\{1, 3, d_k, d\}$
is a Diophantine quadruple in $\Z\left[\sqrt{-2}\right]$,
where $d_k$ is given by $\eqref{dk}$, then $d=d_{k-1}$ or $d=d_{k+1}$.
\end{theorem}

From Theorem \ref{trm1} we immediately obtain the following corollary.

\begin{corollary}\label{c1}
The Diophantine pair $\{1, 3\}$ can not be extended to a Diophantine quintuple in $\Z\left[\sqrt{-2}\right]$.
\end{corollary}

The organization of the paper is as follows. In Section \ref{sc2},
assuming $k$ to be minimal integer for which Theorem \ref{trm1} does
not hold, we translate the assumption of Theorem \ref{trm1} into
system of Pellian equations from which recurrent sequences
$\nu_m^{(i)}$ and $\omega_n^{(j)}$ are deduced, intersections of
which give solutions to the system. In Section \ref{sc3} we use a
congruence method introduced by Dujella and Peth\H{o} \cite{DP98} to
determine the fundamental solutions of Pellian equations. In Section
\ref{sc4} we give a lower bound for $m$ and $n$ for which the
sequences $\nu_m^{(i)}$ and $\omega_n^{(j)}$ intersect. In Section
\ref{sc5} we use a theorem of Bennett \cite{B98} to establish an
upper bound for $k$. Remaining cases are examined separately in
Section \ref{sc6} using linear forms in logarithms, Baker-W\"
ustholz theorem \cite{BW93} and the Baker-Davenport method of
reduction \cite{BD69}.

\section{The system of Pellian equations}\label{sc2}

Let $\{1,3,d_k,d\}$ be a Diophantine quadruple in $\Z\left[\sqrt{-2}\right]$ where $k$ is
the minimal integer for which Theorem \ref{trm1}
does not hold. Assume $k\geq 6$. Clearly $d=d_l$ for some $l\geq
0$. Since $d+1$ and $3d+1$ are negative integers and $d_kd+1$ is a positive
integer, it follows that there exist $x, y ,z \in \Z$ such that
\begin{equation}\label{e.110}
d+1=-2x^2,\qquad 3d+1=-2y^2, \qquad d_kd+1=z^2.
\end{equation} 
The system of equations \eqref{e.110} is equivalent to the following system of
Pellian equations
\begin{align}
\label{e.120} z^2+2d_kx^2&=1-d_k \\
\label{e.121} 3z^2+2d_ky^2&=3-d_k
\end{align}
where
\begin{equation}\label{e.122}
d_k+1=-2s_k^2,\qquad 3d_k+1=-2t_k^2,
\end{equation}
for some $s_k, t_k \in \Z$. Note that we may assume $s_k, t_k \in \N$. Conditions
\eqref{e.122} follow from the fact that $\{1,3,d_k\}$ is a
Diophantine triple in $\Z\left[\sqrt{-2}\right]$ and the fact that $d_k+1$ and $3d_k+1$
are negative integers.

The following propositions describe the set of positive integer solutions of equations \eqref{e.120} and \eqref{e.121}.

\begin{proposition}\label{p2}
There exist $i_0\in\N$ and $z_0^{(i)},x_0^{(i)}\in\Z$, $i=1,2,\ldots,i_0$, such that
$\left(z_0^{(i)},x_0^{(i)}\right)$ are solutions
of the equation \eqref{e.120}, which satisfy
$$1\leq z_0^{(i)}\leq\sqrt{-d_k(1-d_k)}, \qquad 1\leq \left|x_0^{(i)}\right|
\leq\sqrt{\frac{1-d_k^2}{2d_k}},$$
and such that for every solution $(z,x)\in \N \times \N$ of the equation \eqref{e.120},
there exists $i\in\{1,2,\ldots,i_0\}$ and an integer $m\geq 0$ such that
$$z+x\sqrt{-2d_k}=\left(z_0^{(i)}+x_0^{(i)}\sqrt{-2d_k}\right)\left(-2d_k -1+2s_k\sqrt{-2d_k}\right)^m.$$
\end{proposition}

\begin{proof}
The fundamental solution of the related Pell's equation
$z^2+2d_kx^2=1$ is $-2d_k-1+2s_k\sqrt{-2d_k}$ since
$$(-2d_k-1)^2+2d_k\cdot (2s_k)^2=4d_k^2+4d_k+1-4d_k(1+d_k)=1$$
and $-2d_k-1>2s_k^2-1=-d_k-2,$ see \cite[Theorem 105]{N91}.
Following arguments of Nagell \cite[Theorem 108]{N91} we
obtain that there are finitely many integer solutions
$\left(z_0^{(i)},x_0^{(i)}\right)$, $i=1,2,\ldots,i_0$ of the equation
\eqref{e.120} such that the following inequalities hold
$$1\leq \left|z_0^{(i)}\right|\leq\sqrt{-d_k(1-d_k)}, \qquad 0\leq \left|x_0^{(i)}\right|\leq\sqrt{\frac{1-d_k^2}{2d_k}},$$
and such that if $z+x\sqrt{-2d_k}$ is a solution of the
equation \eqref{e.120} with $z$ and $x$ in $\Z$, then
$$z+x\sqrt{-2d_k}=\left(z_0^{(i)}+x_0^{(i)}\sqrt{-2d_k}\right)\left(-2d_k-1+2s_k\sqrt{-2d_k}\right)^m$$
for some $m\in \Z$ and  $i\in \{1, 2, \ldots , i_0\}$. Hence
$$z_0^{(i)}+x_0^{(i)}\sqrt{-2d_k}=\left(z+x\sqrt{-2d_k}\right)\left(-2d_k-1+2s_k\sqrt{-2d_k}\right)^{-m},$$
wherefrom it can be easily deduced that if $z+x\sqrt{-2d_k}$ is a
solution of the equation \eqref{e.120} with $z$ and $x$ in $\N$,
then $z_0^{(i)}>0$. Hence 
$$1\leq z_0^{(i)} \leq \sqrt{-d_k(1-d_k)}$$
for all $i\in \{1, 2, \ldots , i_0\}$. If $x_0^{(i)}=0$, we get a contradiction with the upper bound for $z_0^{(i)}$, hence $\left|x_0^{(i)}\right|\geq 1$.
To complete the proof it remains to show that $m\geq 0$. Assume to the contrary that
$m<0$. Then
$$\left(-2d_k-1+2s_k\sqrt{-2d_k}\right)^m=\alpha-\beta\sqrt{-2d_k}$$
with $\alpha,\beta\in\N$ and $\alpha^2+2d_k\beta^2=1$. Since
$$z+x\sqrt{-2d_k}=\left(z_0^{(i)}+x_0^{(i)}\sqrt{-2d_k}\right)\left(\alpha-\beta \sqrt{-2d_k}\right),$$
we have $x=-z_0^{(i)}\beta +x_0^{(i)} \alpha$. By  squaring $x_0^{(i)}
\alpha=x+z_0^{(i)}\beta$ and substituting $\alpha^2=1-2d_k\beta^2$ we get
$$\left(x_0^{(i)}\right)^2=\beta^2(1-d_k)+x^2+2xz_0^{(i)}\beta>\beta^2(1-d_k)\geq 1-d_k>\frac{1-d_k^2}{2d_k},$$
since $x$, $z_0^{(i)}$, $\beta$ and $k $ are positive integers. This is in contradiction with the upper bound for $x_0^{(i)}$.
\end{proof}

Using same arguments we can prove the following proposition.

\begin{proposition}\label{p3}
There exists $j_0\in\N$ and $z_1^{(j)},y_1^{(j)}\in\Z$, $j=1,2,\ldots,j_0$, such that
$\left(z_1^{(j)},y_1^{(j)}\right)$ are solutions of the equation \eqref{e.121}, which satisfy
$$1\leq z_1^{(j)} \leq \sqrt{-d_k(3-d_k)},\qquad 1\leq \left|y_1^{(j)}\right| \leq \sqrt{\frac{(3-d_k)(1+3d_k)}{2d_k}},$$
and such that for every solution $(z,y)\in \N \times \N$ of the equation \eqref{e.121},
 there exists $j\in\{1,2,\ldots,j_0\}$ and an integer $n\geq 0$ such that $$z\sqrt{3}+y\sqrt{-2d_k}=\Big(z_1^{(j)}\sqrt{3}+y_1^{(j)}\sqrt{-2d_k} \Big)\Big(-6d_k-1+2t_k\sqrt{-6d_k}\Big)^n.$$
\end{proposition}\qed

Finitely many solutions that satisfy bounds given in Proposition
\ref{p2} and Proposition \ref{p3} will be called {\it fundamental}
solutions.

From Proposition \ref{p2} and Proposition \ref{p3} it follows that
if $(z,x)$ is a solution in positive integers  of the equation \eqref{e.120},
then $z=\nu_m^{(i)}$ for some $m\geq 0$ and $i\in\{1,2,\ldots,i_0\}$, where
\begin{align}\label{niz1}
\nonumber \nu_0^{(i)}&=z_0^{(i)},\\
\nonumber \nu_1^{(i)}&=(-2d_k-1)z_0^{(i)}-4s_kd_kx_0^{(i)},\\
\nu_{m+2}^{(i)}&=(-4d_k-2)\nu_{m+1}^{(i)}-\nu_m^{(i)},
\end{align}
and if $(z,y)$ is a solution in positive
integers of the equation \eqref{e.121}, then $z=\omega_n^{(j)}$ for some $n\geq 0$ and
$j\in\{1,2,\ldots,j_0\}$, where
\begin{align}\label{niz2}
\nonumber \omega_0^{(j)}&=z_1^{(j)},\\
\nonumber \omega_1^{(j)}&=(-6d_k-1)z_1^{(j)}-4t_kd_ky_1^{(j)},\\
\omega_{n+2}^{(j)}&=(-12d_k-2)\omega_{n+1}^{(j)}-\omega_n^{(j)}.
\end{align}
Therefore, we are looking for the intersection of
sequences $\nu_m^{(i)}$ and $\omega_n^{(j)}$.

\section{Congruence method}\label{sc3}
Using the congruence method introduced by Dujella and Peth\H{o}
\cite{DP98} we determine the fundamental solutions of the equations
\eqref{e.120} and \eqref{e.121}.

\begin{lemma}\label{l1}
$$\nu_{2m}^{(i)}\equiv z_0^{(i)} \pmod {-2d_k}, \qquad \nu_{2m+1}^{(i)}\equiv -z_0^{(i)} \pmod {-2d_k},$$
$$\omega_{2n}^{(j)}\equiv z_1^{(j)}\pmod {-2d_k}, \qquad \omega_{2n+1}^{(j)}\equiv -z_1^{(j)}\pmod { -2d_k},$$
for all $m, n\geq 0$, $i\in \{1, 2, \ldots , i_0\}$,  $j\in \{1, 2, \ldots , j_0\}$.
\end{lemma}

\begin{proof} 
Easily follows by induction. 
\end{proof}

\begin{lemma}\label{l2}
If $\nu_{m}^{(i)}=\omega_{n}^{(j)}$ for some $m, n\geq 0$,
$i\in \{1, 2, \ldots , i_0\}$,  $j\in \{1, 2, \ldots , j_0\}$, then $z_0^{(i)}=z_1^{(j)}$ or $z_0^{(i)}+z_1^{(j)}=-2d_k.$
\end{lemma}

\begin{proof}
From Lemma \ref{l1} it follows that either $z_0^{(i)}\equiv z_1^{(j)}\pmod {-2d_k}$ or $
z_0^{(i)}\equiv -z_1^{(j)} \pmod {-2d_k}$. In the latter case $z_0^{(i)}+z_1^{(j)}\equiv 0\pmod
{-2d_k}$. From Proposition \ref{p2} and Proposition \ref{p3} we get
\begin{align*}
0&<z_0^{(i)}+z_1^{(j)}\leq\sqrt{-d_k(1-d_k)}+\sqrt{-d_k(3-d_k)}\\
&<-d_k+1-d_k+2=-2d_k+3,
\end{align*}
wherefrom it follows that $z_0^{(i)}+z_1^{(j)}=-2d_k$. 
If $z_0^{(i)}\equiv z_1^{(j)}\pmod{-2d_k}$ and $z_0^{(i)}> z_1^{(j)}$, then 
$$0<z_0^{(i)} -z_1^{(j)}<z_0^{(i)}\leq \sqrt{-d_k(1-d_k)}<-2d_k,$$ 
contradiction. Analogously, if $z_1^{(j)}>z_0^{(i)}$, then 
$$0<z_1^{(j)}-z_0^{(i)} <z_1^{(j)}\leq \sqrt{-d_k(3-d_k)}<-2d_k,$$
contradiction.
\end{proof}

\begin{lemma}\label{l3}
\begin{align}
\nu_m^{(i)}&\equiv (-1)^m \left(z_0^{(i)}+2d_km^2z_0^{(i)}+4d_ks_kmx_0^{(i)}\right)\pmod{8d_k^2}\\
\omega_n^{(j)}&\equiv (-1)^n \left(z_1^{(j)}+6d_kn^2z_1^{(j)}+4d_kt_kny_1^{(j)}\right)\pmod{ 8d_k^2}
\end{align}
for all $m, n\geq 0$, $i\in \{1, 2, \ldots , i_0\}$,  $j\in \{1, 2, \ldots , j_0\}$.
\end{lemma}

\begin{proof} 
Easily follows by induction.
\end{proof}

\begin{lemma}\label{l4}
If $\nu_m^{(i)}=\omega_n^{(j)}$ for some $m, n\geq 0$,  $i\in \{1, 2, \ldots , i_0\}$,
$j\in \{1, 2, \ldots , j_0\}$, then $m\equiv n \pmod {2}$.
\end{lemma}

\begin{proof} 
If $m$ is even and $n$ odd, then Lemma
\ref{l1} and Lemma \ref{l2} imply $z_0^{(i)}+ z_1^{(j)}=-2d_k$. Lemma
\ref{l3} implies $$z_0^{(i)}+2d_km^2z_0^{(i)}+4d_ks_kmx_0^{(i)}
\equiv -z_1^{(j)}-6d_kn^2z_1^{(j)}-4d_kt_kny_1^{(j)}\pmod
{8d_k^2},$$ wherefrom, by substituting $z_0^{(i)}+ z_1^{(j)}=-2d_k$
and dividing by $2d_k$, we obtain
$$-1+m^2z_0^{(i)}+2s_kmx_0^{(i)}\equiv-3n^2z_1^{(j)}-2t_kny_1^{(j)}\pmod {-4d_k}.$$
Since $d_k$ is always odd, from \eqref{e.120} and \eqref{e.121} we
get that  $z_0^{(i)}$ and $z_1^{(j)}$ are even, hence the last
congruence can not hold. Indeed, on the left side is an odd integer
and on the right side is an even integer, contradiction. If $m$ is odd and $n$ even, 
contradiction can be obtained analogously.
\end{proof}

Therefore, the equations $\nu_{2m}^{(i)}=\omega_{2n+1}^{(j)}$ and
$\nu_{2m+1}^{(i)}=\omega_{2n}^{(j)}$ have no solutions in integers
$m, n\geq 0$, $i\in \{1, 2, \ldots , i_0\}$,  $j\in \{1, 2, \ldots ,
j_0\}$. It remains to examine the cases when $m$ and $n$ are
both even or both odd. In each of those cases we have
$z_0^{(i)}=z_1^{(j)}$. Since
$$\left(z_0^{(i)}\right)^2-1=d_k\left(-2\left(x_0^{(i)}\right)^2-1\right),$$
it follows that
$$\delta:=\frac{\left(z_0^{(i)}\right)^2-1}{d_k}$$ is an integer. Furthermore,
$$ \delta+1=-2\left(x_0^{(i)}\right)^2, \qquad 3\delta+1=-2\left(y_1^{(j)}\right)^2, \qquad \delta d_k+1=\left(z_0^{(i)}\right)^2.$$
Thus $\delta$ satisfies system \eqref{e.110} and hence
$\delta=d_l$ for some $l\geq 0$. Moreover, $\{1, 3, d_k, d_l\}$ is a
Diophantine quadruple in $\Z\left[\sqrt{-2}\right]$ since $d_l\neq d_k$.
Indeed, if $d_l=d_k$ then 
$$d_k^2+1=\left(z_0^{(i)}\right)^2,$$
contradiction with $d_k^2\equiv 1 \pmod {4}$. In what follows we show that $l=k-1$.
Assume $\delta>d_{k-1}$, that is $l<k-1$.
Then the triple $\{1, 3, d_l\}$ can be extended to a Diophantine
quadruple in $\Z\left[\sqrt{-2}\right]$ by $d_k$, which differs from $d_{l-1}$
and $d_{l+1}$ since $l-1<l+1<k$ by assumption; this contradicts
the minimality of $k$. Therefore $l\geq k-1$. On the other hand, since
$$\delta d_k+1=\left(z_0^{(i)}\right)^2\leq -d_k(-d_k+1),$$
from Proposition
\ref{p2} it follows that $\delta=d_l> d_k -1$ and hence $l\leq k$.
Since $d_l\neq d_k$ we have $d_l=d_{k-1}$. Hence
$$\left(z_0^{(i)}\right)^2=d_kd_{k-1}+1.$$
From \eqref{e.100} we obtain $z_0^{(i)}=z_0=c_{k-1}+2.$
Furthermore, from \eqref{e.120},
\eqref{e.121} and \eqref{e.122} we get $\left|{x_0^{(i)}}\right|=s_{k-1}$ and
$\left|{y_1^{(j)}}\right|=t_{k-1}$. Moreover, from
\begin{align*}
s_k&=\frac{1}{2\sqrt{3}}\left(\left(2+\sqrt{3}\right)^k-\left(2-\sqrt{3}\right)^k\right),\\
t_k&=\frac{1}{2}\left(\left(2+\sqrt{3}\right)^k+\left(2-\sqrt{3}\right)^k\right),
\end{align*}
we get
\begin{equation}\label{e.150}
2s_ks_{k-1}=c_{k-1},\qquad 2t_kt_{k-1}=3c_{k-1}+4.
\end{equation} 
This brings us to the important conclusion. If the system of Pellian
equations \eqref{e.120} and \eqref{e.121} has a solution in positive integers, where $k$ is the smallest integer
for which Theorem \ref{trm1} does not hold and under assumption $k\geq 6$, the fundamental
solutions of Pellian equations \eqref{e.120} and \eqref{e.121} are
$(z_0, x_0^{\pm})$ and  $(z_1, y_1^{\pm})$ respectively, where
\begin{equation}\label{fun1}
z_0=z_1=2(s_ks_{k-1}+1),
\end{equation}
\begin{equation}\label{fun2}
x_0^{\pm}=\pm s_{k-1},\quad y_1^{\pm}=\pm t_{k-1}.
\end{equation}

\section{The lower bound for $m$ and $n$}\label{sc4}

After plugging \eqref{fun1} and \eqref{fun2} into \eqref{niz1} and \eqref{niz2} and expanding we get
\begin{align*}
\nu_m^{\pm}&=\frac{1}{2} \left(2(s_ks_{k-1}+1)  \pm s_{k-1}\sqrt {-2d_k}\right) \left(-2d_{k}-1 +
2s_k \sqrt{-2d_k}\right)^m \\
&+\frac{1}{2} \left(2(s_ks_{k-1}+1) \mp s_{k-1}
\sqrt{-2d_k}\right)\left(-2d_{k}-1 - 2s_k \sqrt{-2d_k}\right)^m,
\end{align*}
and
\begin{align*}
\omega_n^{\pm}&=\frac{1}{2\sqrt{3}} \left(2(s_ks_{k-1}+1)\sqrt{3} \pm
t_{k-1}\sqrt{-2d_k}\right)\left(-6d_{k}-1+ 2t_k\sqrt{-6d_k}\right)^n\\
&+\frac{1}{2\sqrt{3}} \left(2(s_ks_{k-1}+1)\sqrt{3} \mp
t_{k-1}\sqrt{-2d_k}\right) \left(-6d_{k}-1- 2t_k\sqrt{-6d_k}\right)^n,
\end{align*}
for $m, n\geq 0$. One intersection of these sequences is clearly
$$\nu_0^{\pm}=\omega_0^{\pm}=2(s_ks_{k-1}+1),$$ 
wherefrom it follows that the triple $\{1,3,d_k\}$ can be extended to a
Diophantine quadruple in $\Z\left[\sqrt{-2}\right]$ by $d_{k-1}$. Another intersection
is $\nu_1^{-}=\omega_1^{-}$. Indeed, \eqref{e.150} implies
\begin{equation}\label{e.151}
s_ks_{k-1}+1=\frac{1}{3} \left(t_kt_{k-1}+1\right)
\end{equation} 
and hence
\begin{align*}
\omega_1^{-}&=-2-12d_k-2s_{k}s_{k-1}-12d_ks_{k}s_{k-1}+4d_kt_{k}t_{k-1}\\
&=-2-4d_k-2s_{k}s_{k-1}=\nu_1^{-}.
\end{align*}
Therefrom it follows that the triple $\{1,3,d_k\}$ can be extended to a
Diophantine quadruple in $\Z\left[\sqrt{-2}\right]$ by $d_{k+1}$.
Using \eqref{e.151} we can write $\omega_n^{\pm}$ as follows
\begin{align*}
\omega_n^{\pm}&=\frac{1}{6} \left(2(t_kt_{k-1}+1) \pm t_{k-1}\sqrt{-6d_k}\right)
\left(-6d_{k}-1+ 2t_k\sqrt{-6d_k}\right)^n\\
&+\frac{1}{6} \left(2(t_kt_{k-1}+1) \mp
t_{k-1}\sqrt{-6d_k}\right) \left(-6d_{k}-1- 2t_k\sqrt{-6d_k}\right)^n.
\end{align*}
Since
\begin{eqnarray*}
2(s_ks_{k-1}+1)-s_{k-1}\sqrt
{-2d_k}&=&2-\frac{\sqrt{-2d_{k-1}-2}}{\sqrt{-2d_k-2}+\sqrt{-2d_k}}\\
&>&2-\frac{\sqrt{-2d_{k}-2}}{\sqrt{-2d_k-2}+\sqrt{-2d_k}}>1,
\end{eqnarray*}
it follows that
$$ \nu_m^{+}\geq\nu_m^{-}>\frac{1}{2} \left(-2d_k-1 + 2s_k\sqrt{-2d_k}\right)^m.$$
Furthermore,
$$\omega_n^{-}\leq\omega_n^{+} < \frac{1}{2} \left(-6d_k-1+2t_k\sqrt{-6d_k}\right)^{n+1},$$
since
\begin{align*}
2(t_kt_{k-1}+1) - t_{k-1}\sqrt {-6d_k} < \left(\frac{-6d_{k}-1 - 2t_k\sqrt{-6d_k}}{-6d_{k}-1+ 2t_k\sqrt{-6d_k}}\right)^n
\end{align*}
and
$$\frac{1}{3} \left(2(t_kt_{k-1}+1) + t_{k-1}\sqrt {-6d_k} +1\right) <-6d_k-1+2t_k\sqrt{-6d_k},$$
which can be easily verified using \eqref{e.122}.
Therefore, if one of the equations $\nu_m^{\pm}=\omega_n^{\pm}$
has solutions, then
$$\frac{1}{2} \left(-2d_k-1 + 2s_k\sqrt{-2d_k}\right)^m<\frac{1}{2} \left(-6d_k-1+2t_k\sqrt{-6d_k}\right)^{n+1},$$
wherefrom
$$\frac{m}{n+1}< \frac{\log \left(-6d_k-1+2t_k\sqrt{-6d_k}\right)}{\log \left(-2d_k-1+2s_k\sqrt{-2d_k}\right)}.$$
The expression on the right side of the inequality decreases when
$k$ increases. Since $k\geq 6$ it follows that 
$$\frac{m}{n+1} < 1.072.$$
We may assume $n\geq 2$. Indeed for $n=1$ we have
$m\leq 2$ and since $m$ and $n$ are both even or both odd it follows
that the only possibility is $m=1$. We have already established the
intersection $\nu_1^{-}=\omega_1^{-}$ and  it can be easily verified
that $\nu_1^{+}\not=\omega_1^{\pm}$ and
$\nu_1^{-}\not=\omega_1^{+}$. Now it can be easily deduced that
$m<n\sqrt{3}$. Hence, if the sequences $(\nu_m^{\pm})$ and
$(\omega_n^{\pm})$ have any intersections besides two already
established ones, then $n\geq 2$, $m$ and $n$ are of the same parity and
$m<n\sqrt{3}$. We further on assume these conditions.

\begin{proposition}
Let $n\geq 2$. If one of the equations  $\nu_m^{\pm}=\omega_n^{\pm}$
has solutions then 
$$m\geq n \geq \frac{2}{3}\cdot \sqrt[4]{-d_k}.$$
\end{proposition}

\begin{proof} 
If $m<n$, then $m\leq n-2$,
since $m$ and $n$ are of the same parity. From \eqref{niz1} and
\eqref{niz2} using \eqref{e.150} one easily finds $\nu_0^{+}<
\omega_2^{-}$. It can be shown by induction that
$\nu_m^{+}<\omega_{m+2}^{-}$ for $m\geq 0$. Indeed, sequences
$(\nu_m^{\pm})$ and $(\omega_n^{\pm})$ are strictly increasing
positive sequences, which can be easily checked by induction after plugging
\eqref{fun1} and \eqref{fun2} into \eqref{niz1} and \eqref{niz2}.
Hence 
$$\nu_{m+1}^{+}<(-4d_k-2)\nu_m^{+},\qquad \omega_{m+3}^{-}>(-12d_k-3)\omega_{m+2}^{-}.$$
Then clearly $\nu_m^{+}<\omega_{m+2}^{-}$ implies
$\nu_{m+1}^{+}<\omega_{m+3}^{-}$, which completes the proof by
induction. Since 
$$\nu_m^{-}\leq\nu_m^{+}<\omega_{m+2}^{-}\leq\omega_{m+2}^{+},$$
it follows that if one of the equations
$\nu_m^{\pm}=\omega_n^{\pm}$ has solutions, then $m+2>n$, a
contradiction. Hence $m\geq n$. For the second part of the statement assume to the contrary that
$n<\frac{2}{3}\sqrt[4]{-d_k}$. Let us show how we can reach a contradiction in the case 
$\nu_m^{+}=\omega_n^{+}$. Other three case
can be similarly resolved. 

Since $m$ and $n$ are of the same parity,
Lemma \ref{l3} implies that if  $\nu_m^{+}=\omega_n^{+}$, then
\begin{equation}\label{e.200}
(c_{k-1}+2) (m^2-3n^2+m-3n)\equiv 2(m-n) \pmod {-4d_k},
\end{equation}
and since \eqref{e.100} implies $(c_{k-1}+2)^2\equiv 1 \pmod {-d_k},$
we obtain
$$(m^2-3n^2+m-3n)^2\equiv 4 (m-n)^2 \pmod {-d_k}.$$
Moreover
\begin{equation}\label{e.202}
(m^2-3n^2+m-3n)^2 \equiv 4(m-n)^2 \pmod {-4d_k}
\end{equation}
since $(4, d_k)=1$ and both sides of the congruence relation are
divisible by 4, since $m$ and $n$ are of the same parity. Under assumption
$n<\frac{2}{3}\sqrt[4]{-d_k}$ one easily sees that the expressions
on both sides of the congruence relation \eqref{e.202} are strictly
smaller than $-4d_k$. Indeed, 
\begin{align*}
0 \leq 2(m-n)\leq 2n\left(\sqrt{3}-1\right)
<2\left(\sqrt{3}-1\right)\frac{2}{3}\sqrt[4]{-d_k}<\sqrt{-4d_k}
\end{align*}
and 
\begin{align*}
0<-m^2+3n^2-m+3n\leq 2n^2+2n\leq 3n^2<\frac{12}{9}\sqrt{-d_k}<\sqrt{-4d_k}.
\end{align*}
Therefore $-m^2+3n^2-m+3n =2(m-n)$, wherefrom clearly $m\neq n,$ so $m>n$. From \eqref{e.200} we obtain
$$-(c_{k-1} +2) \cdot 2(m-n)\equiv 2(m-n) \pmod {-4d_k},$$
wherefrom
$$ -2s_ks_{k-1} (m-n) \equiv 3(m-n) \pmod {-2d_k}.$$
Since \eqref{e.122} implies $-2s_k^2\equiv 1 \pmod {-d_k}$, by multiplying both sides of the previous equation by $s_k$
we obtain 
$$s_{k-1}(m-n)\equiv 3s_k(m-n)\pmod {-d_k},$$
and since $2\mid m-n$ and $(d_k, 2)=1$, it follows that
\begin{equation}\label{e.204}(m-n) 
(3s_k-s_{k-1}) \equiv 0\pmod{-2d_k}.
\end{equation}
On the other hand, from
$$0<m-n<n\left(\sqrt{3}-1\right)<\left(\sqrt{3}-1\right)\frac{2}{3}\sqrt[4]{-d_k}<0.49\cdot \sqrt[4]{-d_k}$$
and
$$0< 3s_k - s_{k-1}\leq 3s_k=3\cdot\sqrt{\frac{-d_k-1}{2}}<3 \cdot \sqrt {\frac{-d_k}{2}}$$
it follows that
$$ 0<(m-n) (3s_k-s_{k-1})< 1.04\cdot\sqrt[4]{-d_k^3}<-2d_k.$$
Therefore, we have a contradiction with \eqref{e.204}. Completely
analogously a contradiction can be obtained in other three cases, i.e\@ when
$\nu_m^{+}=\omega_n^{-}$, $\nu_m^{-}=\omega_n^{+}$ and
$\nu_m^{-}=\omega_n^{-}$.
\end{proof}

\section{Application of Bennett's theorem}\label{sc5}

\begin{lemma}\label{l.5} Let
$$\theta_1=\sqrt {1+\frac{1}{d_k}}, \qquad \theta_2=\sqrt {1+\frac{1}{3d_k}}$$
and let $(x, y, z)$ be a solution in positive integers of the system of
Pellian equations \eqref{e.120} and \eqref{e.121}. Then
$$\max \left \{\left |\theta_1-\frac{6s_kx}{3z}\right |, \left|\theta_2-\frac{2t_ky}{3z}\right | \right \}<(1-d_k)z^{-2}.$$
\end{lemma}

\begin{proof}
Clearly $\theta_1=\frac{2s_k}{\sqrt{-2d_k}}$
and $\theta_2=\frac{2t_k}{\sqrt{-6d_k}}$. Hence,
\begin{align*}
\left|\theta_1-\frac{6s_kx}{3z}\right| &= \left|\frac{2s_k}{\sqrt{-2d_k}}-\frac{2s_kx}{z}\right|= 2s_k\left|
\frac{z-x\sqrt{-2d_k}}{z\sqrt{-2d_k}}\right| \\
&=\frac{2s_k}{z\sqrt{-2d_k}}\cdot \frac{1-d_k}{z+x\sqrt{-2d_k}}<\frac {2s_k(1-d_k)}{\sqrt{-2d_k}}\cdot z^{-2}\\
&<(1-d_k)\cdot z^{-2}.
\end{align*}
and
\begin{align*}
\left|\theta_2-\frac{2t_ky}{3z}\right| &= \left|\frac{2t_k}{\sqrt{-6d_k}}-\frac{2t_ky}{3z}\right|=
\frac{2t_k}{\sqrt{3}}\left|\frac{z\sqrt{3}-y\sqrt{-2d_k}}{z\sqrt{-2d_k}\sqrt{3}}\right|\\
&=\frac{2t_k}{3z\sqrt{-2d_k}}\cdot \frac{3-d_k}{z\sqrt{3}+y\sqrt{-2d_k}}\\ &<\frac
{2t_k(3-d_k)}{3\sqrt{-6d_k}}\cdot z^{-2}<\frac{3-d_k}{3}\cdot z^{-2}<(1-d_k)\cdot z^{-2}.
\end{align*}
\end{proof}

In order to establish the lower bound for the expression in Lemma \ref{l.5} 
we use the following result of Bennett \cite{B98} on simultaneous rational
approximations of square roots of rationals which are close to $1$.

\begin{theorem}\label{Bennett}
If $a_i, p_i, q$ and $N$ are integers for $0\leq i \leq 2$ with
$a_0<a_1<a_2$, $a_j=0$ for some $0\leq j \leq 2$, $q$ nonzero and
$N>M^9$ where
$$M=\max \{|a_i|:0\leq i \leq 2\},$$
then we have
$$\max_{0\leq i \leq 2} \left \{\left|\sqrt{1+\frac{a_i}{N}} - \frac{p_i}{q}\right |\right \}>(130N \gamma)^{-1}q^{-\lambda}$$
where
$$\lambda = 1+\frac{\log (33N\gamma)}{\log\left (1.7N^2 \prod_{0\leq i<j\leq 2}(a_i-a_j)^{-2}\right)}$$
and
\begin{center}
$\gamma=\left\{ \begin{array}{r@{\,,\ }l}
\frac{(a_2-a_0)^2(a_2-a_1)^2}{2a_2-a_0-a_1} & \quad a_2-a_1\geq a_1-a_0\\
\frac{(a_2-a_0)^2(a_1-a_0)^2}{a_1+a_2-2a_0} & \quad a_2-a_1<a_1-a_0.
\end{array} \right.$
\end{center}
\end{theorem}

We can apply Theorem \ref{Bennett} with
\begin{align*}
N&=-3d_k, & a_0&= -3, & a_1&=-1, & a_2&=0,\\
M&=3, & q&=3z, & p_1&=6s_kx, & p_2&=2t_ky,
\end{align*}
since $N=-3d_k>3^9$ for $k\geq 6$.
So,
$$\max \left \{ \left |\theta_1-\frac{6s_kx}{3z}\right|, \left |\theta_2-\frac{2t_ky}{3z}\right |\right \}>\left(130\cdot (-3d_k)\gamma\right )^{-1}\cdot (3z)^{-\lambda},$$
where
$$\gamma=\frac{36}{5}, \qquad \lambda=1+\frac{\log\left (-99d_k\cdot \frac{36}{5}\right )}{\log \left (1.7 \cdot 9d_k^2\cdot \frac{1}{36}\right)}.$$
From Lemma \ref{l.5} we get
$$z^{-\lambda +2}<(1-d_k)\left(130\cdot \left(-3d_k\right)\cdot \frac{36}{5}\right)\cdot 3^\lambda.$$
Since $\lambda < 2$ and $-d_k(1-d_k)<1.000000821d_k^2$ for $k\geq 6$, it
follows that $z^{-\lambda +2}<25272.03d_k^2$
and hence
$$(-\lambda +2) \log z< \log \left(25272.03d_k^2\right).$$
Since
$$ \frac{1}{2-\lambda}=
\frac{1}{1-\frac{\log\left(-99d_k\cdot \frac{36}{5}\right)}{\log \left(1.7\cdot
9d_k^2\cdot \frac{1}{36}\right)}}\leq \frac{\log \left(0.425d_k^2\right)}{\log
(-0.00059d_k) }$$ we have
\begin{equation}\label{e.500}
\log z < \frac{\log \left(25272.03 d_k^2\right) \log \left(0.425d_k^2\right)}{\log
(-0.00059d_k)}. 
\end{equation} 
Furthermore, since $z=\nu_m^{\pm}$ for
some $m\geq 0$, it follows that
$$z>\frac{1}{2} \left(-2d_k-1+2s_k\sqrt{-2d_k}\right)^m.$$
Since $2s_k\sqrt{-2d_k} > -2d_k-2$ for $k\geq 0$ it follows that
$$z>\frac{1}{2} \left(-4d_k-3\right)^m.$$
From  $(-4d_k-3)^{-1}<\frac{1}{2}$ for $k\geq 1$, we get $z>(-4d_k-3)^{m-1}.$
Therefore,
$$\log z > (m-1) \log (-4d_k-3),$$
and since $m\geq n \geq \frac{2}{3}\cdot \sqrt[4]{-d_k}$, it follows that
$m-1>0.5\cdot \sqrt[4]{-d_k}$ and hence
$$\log z > 0.5\cdot \sqrt[4]{-d_k}\cdot \log (-4d_k-3).$$
Using \eqref{e.500} we obtain
$$\sqrt[4]{-d_k}<\frac{\log \left(25272.03 d_k^2\right) \log \left(0.425d_k^2\right)}{0.5\cdot
\log (-0.00059d_k) \log (-4d_k-3)}.$$
The expression on the right side of the inequality decreases when
$k$ increases, and hence by substituting $k=6$ we obtain
$$\sqrt[4]{-d_k}<20.477$$
and finally
$$-d_k<175\,817.$$
This implies $k\leq 5$, which contradicts the assumption $k\geq 6$.
Therefore, the minimal integer $k$ for which Theorem \ref{trm1} does
not hold, if such exists, is smaller than $6$.

\section{Small cases}\label{sc6}
To complete the proof it remains to show that Theorem \ref{trm1} holds also for $0\leq k\leq 5$.
In each case we have to solve a system of Pellian equations where
one of the equations is always the Pell's equation 
$$y^2-3x^2=1$$
and the second one is as follows
\begin{itemize}
\item if $k=0$ \quad $z^2-2x^2=2,$
\item if $k=1$ \quad $z^2-6x^2=4$,
\item if $k=2$ \quad $z^2-22y^2=12$,
\item if $k=3$ \quad $z^2-902x^2=452$,
\item if $k=4$ \quad $z^2-4182y^2=2092$ ,
\item if $k=5$ \quad $z^2-58242y^2=29122$.
\end{itemize}
All the solutions in positive integers of $y^2-3x^2=1$ are given by
$(x,y)=(x_m',y_m')$, where
\begin{align*}
x_m'&=\frac{1}{2\sqrt{3}}\left((2+\sqrt{3})^m-(2-\sqrt{3})^m\right),\\
y_m'&=\frac{1}{2}\left((2+\sqrt{3})^m+(2-\sqrt{3})^m\right)
\end{align*}
and $m\geq0$.  Likewise, we can find a sequence of solutions for any of the 
equations listed above. The above systems can be reduced to finding the
intersections of $(x_m')$ and following sequences:
\begin{align*}
k=0: \ \ x_n=&\frac{1+\sqrt{2}}{2}(3+2\sqrt{2})^n+\frac{1-\sqrt{2}}{2}(3-2\sqrt{2})^n,\\
k=1: \ \ x_n=&\frac{1}{\sqrt{6}}(5+2\sqrt{6})^n-\frac{1}{\sqrt{6}}(5-2\sqrt{6})^n,\\
k=3: \ \ x_n^{\pm}=&\pm\frac{61+2\sqrt{902}}{\sqrt{902}}(901\pm30\sqrt{902})^n,\\
&\mp\frac{61-2\sqrt{902}}{\sqrt{902}}(901\mp 30\sqrt{902})^n
\end{align*}
that is  to finding the intersections of $(y_m')$ and following sequences:
\begin{align*}
k=2: \ \  y_n^{\pm}=&\pm\frac{5+\sqrt{22}}{\sqrt{22}}(197\pm42\sqrt{22})^n
\mp\frac{5-\sqrt{22}}{\sqrt{22}}(197\mp 42\sqrt{22})^n,\\
k=4: \ \ y_n^{\pm}=&\pm\frac{841+13\sqrt{4182}}{\sqrt{4182}}(37637\pm582\sqrt{4182})^n
\mp\\
&\mp\frac{841-13\sqrt{4182}}{\sqrt{4182}}(37637\mp582\sqrt{4182})^n,\\
k=5: \ \ y_n^{\pm}=&\pm\frac{23419+97\sqrt{58241}}{2\sqrt{58241}}(524177\pm2172\sqrt{58241})^n\mp\\
&\mp\frac{23419-97\sqrt{58241}}{2\sqrt{58241}}(524177\mp2172\sqrt{58241})^n,
\end{align*}
with $n\geq 0$. In what follows, we will briefly resolve the case $k=1$, 
so to demonstrate a method based on Baker's theory on linear forms in logarithms. 

If $k=1$ the problem reduces to finding the intersection of sequences
\begin{align*}
x_m'&=\frac{1}{2\sqrt{3}}\left((2+\sqrt{3})^m-(2-\sqrt{3})^m\right)\\
x_n&=\frac{1}{\sqrt{6}}\left((5+2\sqrt{6})^n-(5-2\sqrt{6})^n\right)
\end{align*}
Clearly $x_0'=x_0=0$ and $x_2'=x_1=4$. We have to
show that there are no other intersections. Assume $m,n\geq 3$ and $x_m'=x_n$. Setting
$$P=\frac{1}{2\sqrt{3}}(2+\sqrt{3})^m,\quad Q=\frac{1}{\sqrt{6}}(5+2\sqrt{6})^n,$$
we have
$$ P-\frac{1}{12}P^{-1}=Q-\frac{1}{6}Q^{-1}.$$
Since
$$Q-P=\frac{1}{6}Q^{-1}-\frac{1}{12}P^{-1}>\frac{1}{6}(Q^{-1}-P^{-1})=\frac{1}{6}P^{-1}Q^{-1}
(P-Q),$$ we have $Q>P$. Furthermore, from
$$\frac{Q-P}{Q}=\frac{1}{6}Q^{-1}P^{-1}-\frac{1}{12}P^{-2}<
\frac{1}{6}Q^{-1}P^{-1}+\frac{1}{12}P^{-2}<0.25P^{-2}$$ 
we get
\begin{align*}
0<\log\frac{Q}{P} &= -\log\left(1-\frac{Q-P}{Q}\right)<\frac{Q-P}{Q}+\left(\frac{Q-P}{Q}\right)^2\\
&<\frac{1}{4}P^{-2}+\frac{1}{16}P^{-4}<0.32P^{-2}<e^{-m}.
\end{align*}
The expression $\log\frac{Q}{P}$ can be written as a linear form in
three logarithms in algebraic integers. Indeed
$$\Lambda:=\log\frac{Q}{P}=-m\log\alpha_1+n\log\alpha_2+\log\alpha_3,$$
with $\alpha_1=2+\sqrt{3}$, $\alpha_2=5+2\sqrt{6}$ and $\alpha_3=\sqrt{2}$.
Then $0<\Lambda<e^{-m}.$

Now, we can apply the famous result of Baker and W\"{u}stholz \cite{BW93}.

\begin{lemma}\label{BWlemma}
If $\Lambda=b_1\alpha_1+\cdots+b_l\alpha_l\not =0$, where
$\alpha_1,\ldots,\alpha_l$ are algebraic integers and
$b_1,\ldots,b_l$ are rational integers, then
$$ \log|\Lambda|\geq
-18(l+1)!l^{l+1}(32d)^{l+2}h'(\alpha_1)\cdots
h'(\alpha_l)\log(2ld)\log B, $$ where
$B=\max\{|\alpha_1|,\ldots,|\alpha_l|\}$, $d$ is the degree of the
number field generated by $\alpha_1,\ldots,\alpha_l$ over $\Q$,
$$h'(\alpha)=\frac{1}{d}\max\{h(\alpha),|\log\alpha|,1\}$$
and $h(\alpha)$ denotes the logarithmic Weil height of $\alpha$ .
\end{lemma}

In our case $l=3$, $d=4$, $B=m$, $\alpha_1=2+\sqrt{3}$, $\alpha_2=5+2\sqrt{6}$ and $\alpha_3=\sqrt{2}$. 
From Lemma \ref{BWlemma} and from $\Lambda<e^{-m}$ we obtain
$$ m\leq 2\cdot 10^{14}\log m.$$
Since the previous inequality does not hold for $m\geq M=10^{16}$,
we conclude that if there is a  solution of $x_m'=x_n$ then $n\leq
m<M=10^{16}$. This upper bound can be reduced by using
the following lemma, which was originally introduced in \cite{BD69}.
\begin{lemma}[\cite{D99}, Lemma 4a]\label{l-red}
Let $\theta$, $\beta$,$\alpha$, $a$ be a positive real numbers and
let $M$ be a positive integer. Let $p/q$ be a convergent of the
continued fraction expansion of $\theta$ such that $q>6M$. If
$\varepsilon=\|\beta q\|-M\cdot\|\theta q\|>0$, where $\|\cdot\|$
denotes the distance from the nearest integer, then the inequality
$$ |m\theta-n+\beta|<\alpha a^{-m}, $$
has no integer solutions $m$ and $n$ such that $\log(\alpha
q/\varepsilon)/\log a\leq m\leq M$.
\end{lemma}

After we apply Lemma \ref{l-red} with
$\theta=\log\alpha_1/\log\alpha_2$,
$\beta=\log\alpha_3/\log\alpha_2$, $\alpha=1/\log\alpha_2$,
$M=10^{16}$ and $a=e$, we obtain a new upper bound $M=38$ and by
another application of Lemma \ref{l-red} we obtain $M=7$. By examining all the possibilities, we
prove that the only solutions of $x_m'=x_n$ are $x_0'=x_0=0$ and $x_2'=x_1=4$.

All the other cases can be treated similarly. We get these explicit results.
\begin{align*}
k&=0: & x_0&=x_1'=1\\
k&=1: & x_0&=x_0'=0, \quad x_1=x_2'=4\\
k&=2: & y_0^{+}&=y_1'=2, \quad y_1^{-}=y_3'=26\\
k&=3: & x_0^{+}&=x_2'=4, \quad x_1^{-}=x_4'=56 \\
k&=4: & y_0^{+}&=y_3'=26, \quad y_1^{-}=y_5'=362\\ 
k&=5: & y_0^{+}&=y_4'=97, \quad y_1^{-}=y_6'=1351.
\end{align*}
These can be interpreted in terms of Theorem \ref{trm1}. So, if $0\leq k\leq 5$ and the set $\{1, 3, d_k, d\}$
is a Diophantine quadruple in $\Z\left[\sqrt{-2}\right]$, then $d=d_{k-1}$ or $d=d_{k+1}$, which completes the proof of Theorem \ref{trm1}.

\subsection*{Acknowledgements}
Zrinka Franu\v si\'c was supported by the Ministry of Science, Education and Sports, Republic of
Croatia, grant 037-0372781-2821. Dijana Kreso was supported
by the Austrian Science Fund (FWF): W1230-N13 and NAWI Graz.

\bibliographystyle{amsplain}
\bibliography{Dijana1}

\end{document}